\documentclass[12pt,draft]{article}

\usepackage{mathtools}
\usepackage{amssymb}
\usepackage{amsthm}
\usepackage{xcolor}
\usepackage{tikz}
\usepackage{fullpage}

\DeclareSymbolFont{bbold}{U}{bbold}{m}{n}
\DeclareSymbolFontAlphabet{\mathbbold}{bbold}

\newcommand{\R}{\mathbb{R}}

\newcommand{\1}{\mathbbold{1}}

\renewcommand{\epsilon}{\varepsilon}

\DeclareMathOperator{\calTlim}{\mathcal{T}{\text{-}}lim}
\DeclareMathOperator{\grad}{grad}
\newcommand\RprimeSXY{\rlap{$\phantom R'$}R_{S_{X+Y\mkern-1mu}}}
\newcommand\lowered{\rule{-0.7pt}{1.7ex}}

\newcommand{\Xs}{X_{\rm s}}



\providecommand{\dupa}[2]{\left\langle #1 \mkern1.5mu,\mkern1.5mu #2 \right\rangle}

\newcommand{\from}{\colon}

\let\phi\varphi

\let\leq\leqslant
\let\ge\geqslant
\let\geq\geqslant

\makeatletter
\def\@row#1,{#1\@ifnextchar;{\@gobble}{&\@row}}
\def\@matrix{%
    \expandafter\@row\my@arg,;%
    \@ifnextchar({\\ \get@in@paren{\@matrix}}{\after@matrix}%
    }
\def\matrixtype#1#2#3{%
    \ifmmode\def\after@matrix{\end{#2}\right#3}%
    \else\def\after@matrix{\end{#2}\right#3$}$\fi
    \left#1\begin{#2}\get@in@paren{\@matrix}%
    }
\def\@column#1,{#1\@ifnextchar;{\@gobble}{\\ \@column}}
\newcommand\vect{}
\def\svect(#1){\left(\begin{smallmatrix}\@column#1,;\end{smallmatrix}\right)}
\def\vect{\get@in@paren{\@vect}}
\def\@vect{\left(\begin{matrix}\expandafter\@column\my@arg,;\end{matrix}\right)}
\def\get@in@paren#1({\def\my@arg{}\def\my@rest{}\def\after@get{#1}\get@arg}
\let\e@a\expandafter
\def\get@arg#1){\e@a\kl@test\my@rest#1(;}
\def\kl@test#1(#2;{\e@a\def\e@a\my@arg\e@a{\my@arg#1}%
                   \ifx:#2:\let\my@exec\after@get
                   \else\let\my@exec\get@arg
                        \e@a\def\e@a\my@arg\e@a{\my@arg(}%
                        \def@rest#2;%
                   \fi\my@exec}
\def\def@rest#1(;{\def\my@rest{#1\kl@zu}}
\def\kl@zu{)}

\mathchardef\capitaly\mathcode`\Y
\mathcode`\Y=\string"8000

\begingroup \catcode`\Y=\active
  \gdef Y{\capitaly\futurelet\myfirsttoken\testfirsttoken}
\endgroup
\newcommand\testfirsttoken{\ifx\myfirsttoken$\expandafter\getsecondtoken\fi}
\def\getsecondtoken${\futurelet\mysecondtoken\testsecondtoken}
\newcommand\testsecondtoken{\ifx\mysecondtoken.\mkern-2mu\fi
                            \ifx\mysecondtoken,\mkern-2mu$\else$\fi}

\newcommand\MyPairedDelimiter{%
  \@ifstar{\My@Paired@Delimiter{{}}}
          {\My@Paired@Delimiter{}}%
}
\newcommand\My@Paired@Delimiter[4]{%
  \newcommand#2{%
    \@ifstar{\start@PD{#1}{\delimitershortfall=-1sp}{#3}{#4}}
            {\start@PD{#1}{}{#3}{#4}}%
  }%
}
\newcommand\start@PD[5]{%
  #1\mathopen{\mathpalette\put@delim@helper{\put@delim{#2}{#3}{.}{#5}}}%
  #5%
  \mathclose{\mathpalette\put@delim@helper{\put@delim{#2}{.}{#4}{#5}}}%
}
\newcommand\put@delim@helper[2]{%
  \hbox{$\m@th\nulldelimiterspace=0pt #2#1$}%
}
\newcommand\put@delim[5]{%
  \setbox\z@\hbox{$\m@th#5{#4}$}%
  \setbox\tw@\null
  \ht\tw@\ht\z@ \dp\tw@\dp\z@
  #1#5%
  \left#2\box\tw@\right#3%
}

\makeatother
\MyPairedDelimiter*{\abs}{\lvert}{\rvert}
\MyPairedDelimiter*{\norm}{\lVert}{\rVert}
\MyPairedDelimiter{\set}{\{}{\}}

\theoremstyle{plain} 
\newtheorem{theorem}{Theorem}[section]
\newtheorem{corollary}[theorem]{Corollary}
\newtheorem{lemma}[theorem]{Lemma}

\theoremstyle{definition}
\newtheorem{example}[theorem]{Example}

\newtheorem{hypothesis}[theorem]{Hypothesis}
\newtheorem*{definition}{Definition}
\newtheorem{remark}[theorem]{Remark}

\usepackage{enumitem}
\renewcommand{\labelenumi}{{\rm (\alph{enumi})}}

\setenumerate[1]{nolistsep} 
\setenumerate[2]{nolistsep} 

\begin{document}

\medmuskip=4mu plus 2mu minus 3mu
\thickmuskip=5mu plus 3mu minus 1mu
\belowdisplayshortskip=9pt plus 3pt minus 5pt

\title{A note on perturbations of $C_0$-semigroups}

\author{Christian Seifert, Hendrik Vogt and Marcus Waurick}

\date{\today}

\maketitle

\begin{abstract}
This article deals with a variation of constants type inequality for semigroups acting consistently on a scale of Banach spaces. This inequality can be characterized by a corresponding (easy to verify) inequality for their generators. The results have applications to heat kernel estimates and provide a unified perspective to estimates of these type. Moreover, bi-continuous semigroups can be treated as well.
\end{abstract}



\section{Introduction}

In this note, we elaborate on a perturbation result for $C_0$-semigroups. More precisely, we shall further elaborate on a heat kernel estimate given in \cite{bgk2009}, which already has been extended in \cite{w2011}, \cite{SeifertWingert2014} and complemented in \cite{SeifertWaurick2016}. The latter two references provided an analytic approach to the heat kernel type estimates, by means of showing suitable estimates for positive $C_0$-semigroups acting in suitable $L_p$-spaces. 

The main contribution in \cite{SeifertWaurick2016} on this matter was to understand that the desired estimates are in fact consequences of variation of a constants type inequality for perturbations of $C_0$-semigroups. Providing this different point of view, the results, however, did not directly apply to the situation outlined in \cite{SeifertWingert2014}. The main difficulty, which has not been overcome in \cite{SeifertWaurick2016} was that the considered semigroups were acting in different ($L_p$-)spaces, including the non-reflexive space $L_\infty$, where the semigroups treated are \emph{not} strongly continuous anymore.

The main objective of the present work is to resolve this issue and to provide a unified persepective to the results in \cite{SeifertWingert2014} and \cite{SeifertWaurick2016} and, thus, to present a general perspective to these type of perturbation results particularly in non-reflexive Banach spaces. We shall furthermore show that the developed machinery can be applied to bi-continuous semigroups (see \cite{Kuehnemund2001}), which could not be treated within previous frameworks. 

In fact, one can view the main result of the present exposition (Theorem~\ref{thm:main}) as a characterization of a variation of constants type \emph{inequality} for semigroups $S$ and $T$, where $T$ may be considered as a perturbation of $S$, in the case where both $S$ and $T$ are consistently acting on \emph{different} Banach spaces. The unperturbed semigroup $S$ is supposed to satisfy only a rather weak continuity property. This necessitates the introduction of a weak Laplace transform, which in turn coincides with a resolvent of the corresponding generator, if $S$ was strongly continuous.

Although the main focus of applications will be positive semigroups acting on some vector space with compatible order structure, the main result is formulated without the usage of any order structure. Thus, it might also bare applications to a more general context than positive semigroups.

In Section~\ref{sec:semigroups}, we study the weak Laplace transform mentioned above and address some of its properties. The main result can be found in Section~\ref{sec:main_result}. In that section, we shall also elaborate on validating the assumptions of our main result. This article is concluded with a series of applications: We shall employ the main theorem for perturbation results for bi-continuous semigroups; we frame the main results of \cite{SeifertWingert2014,SeifertWaurick2016} into the setting outlined in this paper. Finally, we will treat a perturbation result for delay equations on a scale of $L_p$-spaces, which lead to Banach lattice-valued semigroups -- a class of problems that were out of reach for earlier results.

\section{On Semigroups}\label{sec:semigroups}

We say that $S$ is a \emph{semigroup} on a Banach space $Y$ if $S\colon [0,\infty)\to L(Y)$ satisfies $S(0)=I$ and $S(t+s)=S(t)S(s)$ for all $s,t\in [0,\infty)$.

\begin{lemma}
\label{lem:exponential_bound}
  Let $Y$ be a Banach space, and let $E\subseteq Y'$ be a closed norming subspace for~$Y$. Let $S$ be a semigroup on $Y$ such that
  $[0,\infty)\ni t\mapsto S(t)\in L_{\sigma(Y,E)}(Y)$ is continuous, i.e.,
  \[[0,\infty)\ni t\mapsto \dupa{S(t)y}{e}_{Y\times E}\]
  is continuous for all $y\in Y$, $e\in E$.
  Then there exist $M\geq 1$, $\omega\in\R$, such that $\norm{S(t)}\leq Me^{\omega t}$ for all $t\geq 0$.
\end{lemma}

\begin{proof}
  The proof is almost the same as in the case of $C_0$-semigroups; an additional argument is needed when it comes to the local boundedness at zero. For the latter, we let $(t_n)_n$ in $(0,\infty)$, $t_n\to 0$. Then
  \[\dupa{S(t_n)y}{e}_{Y\times E} \to \dupa{y}{e}_{Y\times E}\]
  for all $y\in Y$, $e\in E$, and by the uniform boundedness principle, $(S(t_n)y)_n$ is bounded in~$E'\mkern-2mu$, for all $y\in Y$.
  Note that this sequence is also a sequence in $Y$, and that it is bounded in $Y$ since $E$ is norming.
  Another application of the uniform boundedness principle yields boundedness of $(S(t_n))_n$. Thus, $S$ is locally bounded at zero. The remainder of the proof is standard.
\end{proof}

\begin{definition}
  Let $Y$ be a Banach space, and let $E\subseteq Y'$ be a closed norming subspace for~$Y$. Let $S$ be a semigroup on $Y$ such that
  $[0,\infty)\ni t\mapsto S(t)\in L_{\sigma(Y,E)}(Y)$ is continuous.
  Let $\lambda\in\R$ be sufficiently large. Then we define $R_S(\lambda)\in L(Y,E')$ by
  \[\dupa{R_S(\lambda) y}{e}_{E'\times E} = \int_0^\infty e^{-\lambda t} \dupa{S(t)y}{e}_{Y\times E}\,dt \quad(y\in Y,\ e\in E).\]
\end{definition}

Note that the right-hand side indeed defines an operator $R_S(\lambda)\in L(Y,E')$, since by Lemma~\ref{lem:exponential_bound} we have
\[\int_0^\infty \abs{e^{-\lambda t} \dupa{S(t)y}{e}_{Y\times E}}\,dt \leq \int_0^\infty e^{-\lambda t} Me^{\omega t} \norm{y}_Y\norm{e}_E\,dt = \frac{M}{\lambda-\omega} \norm{y}_Y\norm{e}_E\]
for $\lambda>\omega$.

\begin{remark}\label{dual-rem}
  \begin{enumerate}
    \item
      Assume that $S$ is strongly continuous on $Y$. Then we can choose $E:=Y'$, and we obtain $R_S(\lambda) = \kappa_Y (\lambda-A_S)^{-1}$,
      where $\kappa_Y\from Y\to Y''$ is the canonical embedding and $A_S$ is the generator of $S$ in $Y$.
    \item
      Let $Y$ be a dual space. Then we choose $E$ to be a predual of $Y$, i.e.\ $E'=Y$, and we identify $E$ with $\kappa_E(E)\subseteq E'' = Y'$, where $\kappa_E\from E\to E''$ is the canonical embedding.
      Then $R_S(\lambda)\in L(Y)$. 
      We assume that $E$ is invariant under the dual semigroup $S' = (S(t)')_{t\ge0}$.
      The continuity of $t\mapsto \dupa{S(t)y}{e}_{Y\times E} = \dupa{y}{S(t)'e}_{Y\times E}$ for all $y\in Y$, $e\in E$ means that $(S(t)'|_{E})_{t\ge0}$ is weakly continuous and hence a $C_0$-semigroup on $E$ by \cite[Theorem~I.5.8]{EngelNagel}. Let $A$ be the generator.
      Then $R_S(\lambda) = \bigl((\lambda-A)^{-1}\bigr)\lowered'$ for sufficiently large $\lambda$.
  \end{enumerate}
\end{remark}

In the proof of our main theorem we will need the following convergence result.
\begin{lemma}
\label{lem:resolvent_convergence}
  Let $Y$ be a Banach space, and let $E\subseteq Y'$ be a closed norming subspace for~$Y$. Let $S$ be a semigroup on $Y$ with
  $[0,\infty)\ni t \mapsto S(t)\in L_{\sigma(Y,E)}(Y)$ continuous.
  Let $y\in Y$, $e\in E$. Then
  \[\dupa{\lambda R_S(\lambda)y}{e}_{E'\times E} \to \dupa{y}{e}_{Y\times E}\quad(\lambda\to\infty).\]
\end{lemma}
Put differently, Lemma~\ref{lem:resolvent_convergence} yields the weak$^*$-convergence $\dupa{\lambda R_S(\lambda)\cdot}{e}_{E'\times E} \to e$ in $Y'$, for all $e\in E$.

\begin{proof}
  We have 
  \begin{align*}
    \dupa{\lambda R_S(\lambda)y}{e}_{E'\times E} & = \int_0^\infty \lambda e^{-\lambda t} \dupa{S(t)y}{e}_{Y\times E}\,dt
    = \int_0^\infty e^{-s} \dupa{S(\tfrac{s}{\lambda})y}{e}_{Y\times E}\,ds \\
    & \to \int_0^\infty e^{-s}\dupa{y}{e}_{Y\times E}\,ds
    = \dupa{y}{e}_{Y\times E} \quad(\lambda\to\infty).
  \end{align*}
  Indeed, as $S$ is exponentially bounded by Lemma~\ref{lem:exponential_bound}, the integral is finite for large $\lambda$, and the dominated convergence theorem applies.
\end{proof}

\section{The perturbation result}
\label{sec:main_result}

In the following, we shall come to our main result on semigroups. For stating the main theorem we need the following notions.

\begin{definition}
  Let $X,Y$ be Banach spaces over $\R$ and $V$ a Hausdorff topological vector space over~$\R$.
  We say that the couple $(X,Y)$ is \emph{$V$-compatible (as Banach spaces)}, provided $X,Y\hookrightarrow V$ continuously.
\end{definition}
We will drop the letter $V$ and simply write \emph{compatible} if there is no risk of confusion.

\begin{remark}
\label{rem:Schnitt_Summe}
  Let $X,Y$ be Banach spaces, $(X,Y)$ compatible. Then $X\cap Y:=\set{v\in V;\; v\in X,\linebreak[1]\ v\in Y}$ 
  endowed with the norm $v\mapsto \max\{\norm{v}_X,\norm{v}_Y\}$ 
  and $X+Y:=\set{v\in V;\; \exists\, x\in X,\linebreak[1]\ y\in Y: v=x+y}$ endowed with the norm
  $x+y\mapsto \inf\{\|x_1\|_X+\|y_1\|_Y;\; x_1\in X,\linebreak[1]\ y_1\in Y,\linebreak[1]\ x+y=x_1+y_1\}$ are Banach spaces as well, see \cite[Lemma 2.3.1]{BerghLoefstroem1976}
\end{remark}

Next we discuss the dual space of $X+Y$.

\begin{remark} 
\label{rem:duals_banach_space}
  Let $X,Y$ be Banach spaces, $(X,Y)$ compatible, and assume that $X\cap Y$ is dense in both $X$ and $Y$.
  Then $(X+Y)'=X'\cap Y'$ by \cite[Theorem 2.7.1]{BerghLoefstroem1976}.
\end{remark}

\begin{lemma}
\label{lem:continuous_in_X'_Y'}
  Let $X,Y,W$ be Banach spaces, $(X,Y)$ compatible, and let $T\from X+Y\to W$ be linear. Then
  \[T|_X\in L(X,W),\ T|_Y\in L(Y,W) \Longleftrightarrow T \in L(X+Y,W).\]
\end{lemma}

\begin{proof}
  Assume that $T|_X\in L(X,W)$ and $T|_Y\in L(Y,W)$.
  Given $v=x+y\in X+Y$, we estimate
  \begin{align*}
    \norm{T(v)} & = \norm{T(x+y)} \leq \norm{T(x)} + \norm{T(y)} \\ 
    & \leq \norm{T}_{L(X,W)} \norm{x}_X + \norm{T}_{L(Y,W)} \norm{y}_Y 
    \leq \max\{\norm{T}_{L(X,W)},\norm{T}_{L(Y,W)}\} \bigl(\norm{x}_X + \norm{y}_Y\bigr).
  \end{align*}
  Taking the infimum over all $x\in X$, $y\in Y$ such that $v=x+y$ we obtain $T\in L(X+Y,W)$.
  
  Conversely, let $T \in L(X+Y,W)$. The canonical embedding $j\from X\to X+Y$, $j(x):=x+0$ is linear and continuous. 
  Since $T|_X = T\circ j$ we have $T|_X\in L(X,W)$, and by symmetry, $T|_Y\in L(Y,W)$. 
\end{proof}

\begin{remark}\label{rem:agree}
  Let $X,Y$ be Banach spaces, $(X,Y)$ compatible.
  Then by definition, the dual pairings $\dupa{\cdot}{\cdot}_{X\times X'}$ and $\dupa{\cdot}{\cdot}_{Y\times Y'}$ agree on $(X\cap Y) \times (X+Y)'$.
  Indeed, for $x\in X\cap Y$ and $v'\in (X+Y)'$ we obtain
  \begin{align*}
    \dupa{x}{v'|_X}_{X\times X'} = v'|_X(x) = v'|_Y(x) = \dupa{x}{v'|_Y}_{Y\times Y'}.
  \end{align*}
  With a slight abuse of notation, we shall write
  \[
    \dupa{x}{v'}_{X\times X'} \coloneqq  \dupa{x}{v'|_X}_{X\times X'}\text{ and } \dupa{x}{v'}_{Y\times Y'} \coloneqq  \dupa{x}{v'|_Y}_{Y\times Y'}
  \]
  and obtain
  \[
    \dupa{x}{v'}_{X\times X'} = \dupa{x}{v'}_{Y\times Y'}.
  \]
\end{remark}

\begin{definition}
  Let $X,Y$ be Banach spaces, $(X,Y)$ compatible, and let $B_X\in L(X)$, $B_Y\in L(Y)$. Then $B_X$ and $B_Y$ are \emph{consistent} if $B_X u = B_Y u$ for all $u\in X\cap Y$.
  Let $T_X$ and $T_Y$ be semigroups on $X$ and $Y$, respectively. We say that $T_X$ and $T_Y$ are \emph{consistent} if $T_X(t)$ and $T_Y(t)$ are consistent for all $t\geq 0$.
\end{definition}

\begin{remark}
  \begin{enumerate}
    \item
      If $T_X$ and $T_Y$ are consistent semigroups on $X$ and $Y$, respectively, then we will also consider the semigroup $T_{X+Y}$ on $X+Y$ determined by
      $T_{X+Y}(t)|_X = T_X(t)$ and $T_{X+Y}(t)|_Y = T_Y(t)$ for all $t\geq 0$; cf.\ Lemma~\ref{lem:continuous_in_X'_Y'}.
    \item
      If $X\cap Y$ is not dense in $X$ and $Y$, then $T_X$ does not need to determine the consistent semigroup $T_Y$ uniquely, and vice versa.
      In applications (see e.g.~\cite{SeifertWingert2014} or the next section, where $X=L_2$, $Y=L_\infty$ for a measure space for which $L_\infty\nsubseteq L_2$), $X\cap Y$ `often' is dense in $X$, but not in $Y$.
  \end{enumerate}
\end{remark}

In view of Lemma~\ref{lem:continuous_in_X'_Y'} we use the following notation:
given a subspace $E \subseteq Y'$, we write $(X+Y)'\cap E$ for the set of all $v'\in(X+Y)'$ satisfying $v'|_Y\in E$.

\begin{remark}\label{rem:w-d}
  Let $X,Y$ be Banach spaces, $(X,Y)$ compatible, and let $E\subseteq Y'$ be a closed norming subspace for~$Y$.
  Let $S$ be a $C_0$-semigroup on $X$ with generator $A_S$, and assume that there exists a semigroup $S_Y$ on $Y$ consistent with $S$, $\dupa{S_Y(\cdot)y}{e}_{Y\times E}$ continuous for all $y\in Y$, $e\in E$.
  Based on the operator $R_{S_Y}(\lambda) \in L(Y,E')$ discussed in Section~\ref{sec:semigroups},
  we define a `dual resolvent' $\RprimeSXY(\lambda) \colon (X+Y)'\cap E \to (X+Y)'$ as follows.
  Let $v'\in(X+Y)'\cap E$.
  We want to define
  \begin{alignat}{2}
    \dupa{\RprimeSXY(\lambda)v'}{x}_{(X+Y)'\times(X+Y)} &:= \dupa{(\lambda-A_S)^{-1}x}{v'|_X}_{X\times X'} \quad && (x\in X), \nonumber \\
    \dupa{\RprimeSXY(\lambda)v'}{y}_{(X+Y)'\times(X+Y)} &:= \dupa{R_{S_Y}(\lambda)y}{v'|_Y\strut}_{E'\times E} \quad && (y\in Y). \label{R'onY}
  \end{alignat}
  For this we need that the two definitions to coincide on $X\cap Y$.
  Indeed, for $x\in X\cap Y$ we obtain with Remark~\ref{rem:agree}
  \begin{align*}
    \dupa{R_{S_Y}(\lambda)x}{v'|_Y}_{E'\times E} & = \int_0^\infty e^{-\lambda t} \dupa{S_Y(t)x}{v'|_Y}_{Y\times E}\,dt
    = \int_0^\infty e^{-\lambda t} \dupa{S(t)x}{v'|_X}_{X\times X'}\,dt\\
    & = \dupa{\int_0^\infty e^{-\lambda t} S(t)x\,dt }{v'|_X}_{X\times X'}
    = \dupa{(\lambda-A_S)^{-1}x}{v'|_X}_{X\times X'}.
  \end{align*}
  We note that $(\RprimeSXY(\lambda)v')|_X = \bigl((\lambda-A_S)^{-1}\bigr)\lowered'(v'|_X)$.
  Moreover, for $v'\in (X+Y)'\cap E$ and $z\in (X+Y)$ we have
  \[\dupa{\RprimeSXY(\lambda)v'}{z}_{(X+Y)'\times(X+Y)} = \int_0^\infty e^{-\lambda t} \dupa{v'}{S_{X+Y}(t) z}_{(X+Y)'\cap E\times (X+Y)}\,dt.\]
\end{remark}

We gather the set of assumptions needed in the our main result, Theorem~\ref{thm:main} below:
\begin{hypothesis}\label{hyp:mainresult}
  Let $X,Y,Z$ be Banach spaces, $V$ a Hausdorff topological vector space. 
  We assume that
  \begin{enumerate}\renewcommand{\labelenumi}{{\rm(\roman{enumi})}}
   \item $(X,Y)$ and $(X,Z)$ are $V$-compatible;
   \item $S$ is a $C_0$-semigroup on $X$ (with generator $A_S$);
   \item $T$ is a $C_0$-semigroup on $X$ (with generator $A_T$);
   \item there exists a $C_0$-semigroup $T_Z$ on $Z$ consistent with $T$;
   \item $E\subseteq Y'$ is a closed, norming subspace;
   \item there exists a semigroup $S_Y$ on $Y$ consistent with $S$ such that $\dupa{S_Y(\cdot)y}{e}_{Y\times E}$ is continuous for all $y\in Y$, $e\in E$.
  \end{enumerate}
\end{hypothesis}
Then we can define the semigroup $S_{X+Y}$ and the `dual resolvent' $\RprimeSXY$ as above.

\begin{theorem}\label{thm:main}
\label{thm:main3} Assume Hypothesis~\ref{hyp:mainresult}. 
  Let $L\subseteq (X+Y)'\cap E$, $L$ invariant under $(S_{X+Y})'$ and under $\lambda \RprimeSXY(\lambda)$,
  $K\subseteq X\cap Z$, $K$ invariant under $T$ and under $\lambda(\lambda-A_T)^{-1}$ for all sufficiently large $\lambda>0$.
  Let $B\in L(Z,Y)$. 
  Then the following are equivalent:
  \begin{enumerate}
    \item
      For all $t\geq 0$ and all $ x \in K$, $v'\in L$ we have
      \[\dupa{T(t)x}{v'}_{X\times X'} \leq \dupa{S(t)x}{v'}_{X\times X'} + \int_0^t \dupa{S_Y(t-s)BT_Z(s)x}{v'}_{Y\times E}\, ds.\]
    \item
      For all sufficiently large $\lambda\in\R$ and all $ x\in K$, $v'\in L$ we have
      \[\dupa{(\lambda-A_T)^{-1}x}{v'}_{X\times X'} \leq \dupa{(\lambda-A_S)^{-1}x}{v'}_{X\times X'} + \dupa{R_{S_Y}(\lambda)B(\lambda-A_{T_Z})^{-1}x}{v'}_{E'\times E}.\]
    \item 
      For all $u\in D(A_T)\cap K$ and all $v'\in L\cap D(A_S')$ we have
      \[\dupa{A_T u}{v'}_{X\times X'} \leq \dupa{u}{A_S'(v'|_X)}_{X\times X'} + \dupa{Bu}{v'}_{Y\times E}.\]
  \end{enumerate}
\end{theorem}

\begin{proof}
  (a) $\Rightarrow$ (b): Let $\lambda\in\R$ be sufficiently large, $x\in K$, $v'\in L$. Then
  \begin{align*}
    \dupa{(\lambda-A_T)^{-1}x - (\lambda-A_S)^{-1}x}{v'}_{X\times X'} & = \int_0^\infty e^{-\lambda t} \dupa{(T(t)x-S(t)x}{v'}_{X\times X'}\,dt \\
    & \leq  \int_0^\infty e^{-\lambda t} \int_0^t \dupa{S_Y(t-s)BT_Z(s)x}{v'}_{Y\times E}\, ds \, dt \\
    & = \int_0^\infty \int_s^\infty e^{-\lambda t} \dupa{S_Y(t-s)BT_Z(s)x}{v'}_{Y\times E}\, dt \, ds \\
    & = \int_0^\infty e^{-\lambda t} \int_0^\infty e^{-\lambda s} \dupa{S_Y(t)BT_Z(s)x}{v'}_{Y\times E}\, ds\, dt  \\
    & = \int_0^\infty e^{-\lambda t}  \dupa{S_Y(t)B\int_0^\infty e^{-\lambda s} T_Z(s)x\, ds}{v'}_{Y\times E}\, dt  \\[0.5ex]
    & = \dupa{R_{S_Y}(\lambda)B (\lambda-A_{T_Z})^{-1}x}{v'}_{E'\times E}.
  \end{align*}
  Note that in the second to last step we pulled the functional $v'\circ S_Y(t)B \in Z'$ out of the integral.

  (b) $\Rightarrow$ (c): 
  Let $u\in D(A_T)\cap K$, $v'\in L\cap D(A_S')$.
  For sufficiently large $\lambda>0$ we define $J_S(\lambda) := \lambda(\lambda-A_S)^{-1}$ and $J_T(\lambda) := \lambda(\lambda-A_T)^{-1}$ and similarly for $J_{T_Z}$.
  Then $A_S J_S(\lambda) = \lambda\bigl(J_S(\lambda) - I\bigr)$, and similarly for $J_T(\lambda)$.
  Hence, 
  \[A_TJ_T(\lambda)u - A_SJ_S(\lambda)u = \lambda\bigl(J_T(\lambda)u-J_S(\lambda)u\bigr).\]
  Thus, by (b),
  \begin{equation}\label{J}
    \dupa{J_T(\lambda)A_T u}{v'}_{X\times X'} - \dupa{ J_S(\lambda)u}{A_S'v'}_{X\times X'} \leq \dupa{\lambda R_{S_Y}(\lambda)BJ_{T_Z}(\lambda) u}{v'}_{E'\times E},
  \end{equation}
  where we wrote $A_S'v' = A_S'(v'|_X)$ for short.
  Next, we let $\lambda\to \infty$. Since $J_S(\lambda),J_T(\lambda)\to I$ strongly in $L(X)$, for the left-hand side we obtain
  \[\dupa{J_T(\lambda)A_T u}{v'}_{X\times X'} - \dupa{ J_S(\lambda)u}{A_S'v'}_{X\times X'} \to \dupa{A_T u}{v'}_{X\times X'} - \dupa{u}{A_S'v'}_{X\times X'}.\]
  By the strong convergence $J_{T_Z}(\lambda)\to I$ in $L(Z)$ we have $BJ_{T_Z}(\lambda)u\to Bu$ in $Y$.
  Moreover, $\dupa{\lambda R_{S_Y}(\lambda)\cdot}{v'}_{E'\times E}\to v'$ weak$^*$ in $Y'$ by Lemma~\ref{lem:resolvent_convergence}.
  Thus, for the right-hand side we obtain
  \[\dupa{\lambda R_{S_Y}(\lambda)BJ_{T_Z}(\lambda) u}{v'}_{E'\times E} \to \dupa{Bu}{v'}_{Y\times E}.\]
  Thus, the limit $\lambda\to\infty$ in~\eqref{J} yields (c).
  
  (c) $\Rightarrow$ (a): 
  Without loss of generality let $t>0$.
  Let $u\in D(A_T)\cap K$ and $v'\in L\cap D(A_S')$.
  Then $T(s)u \in K\cap D(A_T)$ and $S_{X+Y}(t-s)'v' \in L\cap D(A_S')$ for all $s\in[0,t]$ by invariance of $K$ and $L$, respectively; note that
  \[(S_{X+Y}(t-s)'v')|_X = S(t-s)'(v'|_X) \in D(A_S').\]
  Now we use that $A_S'$ is the weak$^*$-generator of~$S'$ and apply~(c) to obtain
  \begin{align*}
    \hspace{2em} & \hspace{-2em}
    \smash[b]{\frac{d}{ds}} \dupa{T(s)u}{S_{X+Y}(t-s)'v'}_{X\times X'} \\[0.5ex]
    &= \dupa{A_T T(s) u}{S_{X+Y}(t-s)'v'\strut}_{X\times X'}
     - \dupa{T(s)u}{A_S'(S_{X+Y}(t-s)'v')|_X\strut}_{X\times X'} \\[0.5ex]
    &\leq \dupa{BT(s)u}{S_{X+Y}(t-s)'v'}_{Y\times E} \\[0.5ex]
    &= \dupa{S_Y(t-s)BT_Z(s)u}{v'}_{Y\times E}.
  \end{align*}
  Therefore,
  \begin{equation}\label{hauptsatz}
  \begin{split}
    \dupa{T(t) u}{v'}_{X\times X'} - \dupa{S(t)u}{v'}_{X\times X'} & = \int_0^t \frac{d}{ds} \dupa{T(s)u}{S_{X+Y}(t-s)'v'}_{X\times X'}\, ds \\
    & \leq \int_0^t \dupa{S_Y(t-s)BT_Z(s)u}{v'}_{Y\times E}\, ds.
  \end{split}
  \end{equation}
 
  Let now $x \in K$, $y'\in L$, and for $\lambda\in\R$ sufficiently large let $J_S(\lambda):=\lambda(\lambda-A_S)^{-1}$ as before, and similarly for $T$.
  Then $u := J_T(\lambda)x \in K\cap D(A_T)$ and $v' := \lambda \RprimeSXY(\lambda)y' \in L\cap D(A_S')$ by invariance of $K$ and $L$, respectively;
  note that $v'|_X = J_S(\lambda)'(y'|_X) \in D(A_S')$ by Remark~\ref{rem:w-d}.
  Applying \eqref{hauptsatz} and~\eqref{R'onY}, we obtain
  \begin{align*}
    \hspace{2em} & \hspace{-2em} \dupa{J_S(\lambda)\bigl(T(t)-S(t)\bigr)J_T(\lambda)x}{y'}_{X\times X'} \\
    & \leq \int_0^t \dupa{S_Y(t-s)BT_Z(s)J_T(\lambda)x}{\lambda \RprimeSXY(\lambda)y'}_{Y\times E}\,ds \\
    & = \int_0^t \dupa{\lambda R_{S_Y}(\lambda)S_Y(t-s)BT_Z(s)J_{T_Z}(\lambda)x}{y'\strut}_{E'\times E}\,ds.
  \end{align*}
  For $\lambda\to\infty$ we obtain the assertion by Lemma~\ref{lem:resolvent_convergence}, the strong convergence $J_{T_Z}(\lambda)\to I$ in $L(Z)$ and dominated convergence. 
\end{proof}

Although the sets $K$ and $L$ can be arbitrary in the above theorem, it turns out that it is no loss of generality that they are closed convex cones.
In fact, for applications (see the concluding section), $K$ and $L$ should be thought of (positive) cones in some Banach lattice.
The positivity as well as the lattice structure, however, does not play the leading role here. That is why we decided to suppress this in the abstract presentation of this work.

\begin{remark}
\label{rem:properties_K_L}
  \begin{enumerate}
    \item
      Without loss of generality we may assume that $K\subseteq X\cap Z$ and $L\subseteq (X+Y)'\cap E$ are convex cones.
      Indeed, the invariance assumptions on the sets $K$ and $L$ imply the invariances for the corresponding convex cones generated by $K$ and $L$, respectively, by linearity.
      Moreover, the statements (a) to~(c) of Theorem~\ref{thm:main3} for the sets $K$ and $L$ imply the corresponding statements for the generated convex cones, again by (bi-)linearity.
    \item
      Without loss of generality we may assume that $K\subseteq X\cap Z$ is closed.
      Indeed, by continuity the invariances of $K$ imply the corresponding invariances of the closure of $K$ in $X\cap Z$.
      Moreover, statement~(b) of Theorem~\ref{thm:main3} implies the corresponding statement for the closure of $K$, again by continuity.
      Thus, by the equivalence of statements (a) to~(c), each of the statements of Theorem~\ref{thm:main3} is valid if and only if it holds with $K$ replaced by its closure.
    \item
      Without loss of generality we may assume that $L\subseteq (X+Y)'\cap E$ is closed in $(X+Y)'$.
      Indeed, by continuity the invariances of $L$ imply the invariances of the closure of $L$ in $(X+Y)'$.
      Moreover, statement~(b) of Theorem~\ref{thm:main3} implies the corresponding statement for the closure of $L$, again by continuity and dominated convergence.
      Again, it follows that each of the statements of Theorem~\ref{thm:main3} is valid if and only if it holds with $L$ replaced by its closure, by continuity and dominated convergence.
    \item
      Let $E$ be invariant under $S_Y'$ and under $\lambda R_{S_Y}(\lambda)'$ for some $\lambda>0$ (by means of embedding $E$ into $E''$ canonically first).
      Then $(X+Y)'\cap E$ is invariant under $(S_{X+Y})'$ and under $\lambda\RprimeSXY(\lambda)$.
      In this case, without loss of generality we may assume that $L\subseteq (X+Y)'\cap E$ is closed in $\sigma((X+Y)'\cap E,X+Y)$. 
      Indeed, let $L$ be invariant under $(S_{X+Y})'$ and under $\lambda\RprimeSXY(\lambda)$ for some $\lambda>0$.      
      By (a) we may assume that $L$ is a convex cone.
      Let 
      \[L^* := \set{z\in X+Y;\;\dupa{e}{z}_{(X+ Y)'\cap E \times (X+Y)} \geq0 \; (e\in L)}\]
      be the dual cone of $L$.
      Let $t\geq 0$.
      If $z\in L^*$, then for all $e\in L$ one has 
      \[\dupa{e}{S_{X+Y}(t) z}_{(X+ Y)'\cap E \times (X+Y)} = \dupa{S_{X+Y}(t)'e}{z}_{(X+Y)'\cap E \times (X+Y)} \geq 0,\]
      since $S_{X+Y}(t)'e\in L$, so it follows that $S_{X+Y}(t)x\in L^*$.
      Thus $L^*$ is invariant under $S_{X+Y}(t)$, hence $L^*$ is invariant under $S_{X+Y}$.
      In the same way one shows that $(L^*)^* = \set{e\in(X+Y)'\cap E;\; \dupa{e}{z}_{(X+Y)'\cap E \times (X+Y)} \geq 0\; (z\in L^*)}$ is invariant under $S_{X+Y}(t)'$, using invariance of $E$.
      Now by the bipolar theorem, $(L^*)^*$ equals the $\sigma((X+Y)'\cap E,X+Y)$-closure of $L$, thus this closure is invariant under $(S_{X+Y})'$.      
      Moreover, for all $e\in \overline{L}$ and $z\in L^*$ we obtain
      \[\dupa{\lambda \RprimeSXY(\lambda)e}{z}_{(X+Y)'\cap E \times (X+Y)} = \int_0^\infty \lambda e^{-\lambda t} \dupa{S_{X+Y}(t)'e}{z}_{(X+Y)'\cap E \times (X+Y)}\, dt \geq 0\]
      since $(S_{X+Y})'(t)e \in \overline{L} = (L^*)^*$.
      This shows that $\lambda \RprimeSXY(\lambda)e \in (L^*)^* = \overline{L}$, i.e.\ $\overline{L}$ is invariant under $\lambda \RprimeSXY(\lambda)$.
      
      We now show that statement~(c) of Theorem~\ref{thm:main3} implies the corresponding statement for the $\sigma((X+Y)'\cap E,X+Y)$-closure of $L$.
      Let $v'\in \overline{L}\cap D(A_S')$. Then there exists $(v_\iota')_\iota$ in $L$ such that $\lim_\iota v_\iota'=v$ in $\sigma((X+Y)'\cap E,X+Y)$. 
      For $\lambda >0$ sufficiently large, we observe that $v_{\iota,\lambda}'\coloneqq \lambda R_{S_{X+Y}}'(\lambda)v_\iota'\in L$, by the invariance of $L$. 
      Moreover, we deduce $v_{\iota,\lambda}|_{X}\in D(A_S')$. Thus, the inequality (c) in Theorem~\ref{thm:main3} holds for $v'$ replaced by $v_{\iota,\lambda}'$.
      By continuity, we can perform the limit in $\iota$ (for the second term note that $A_S' \lambda (\lambda-A_S')^{-1}=\left(A_S \lambda (\lambda-A_S)^{-1}\right)'$ is $\sigma(X',X)$-$\sigma(X',X)$-continuous).
      Letting finally $\lambda\to \infty$, we deduce that (c) of Theorem~\ref{thm:main3} also holds for $v'$ (for the third term use Lemma \ref{lem:resolvent_convergence}).
      Thus, w.l.o.g.\ we may assume that $L$ is $\sigma((X+Y)'\cap E,X+Y)$-closed.
  \end{enumerate}
\end{remark}

We point out that in the above theorem, for each of the sets $K$ and $L$ we assumed two invariances. 
Here we explain the relation between these two invariances.

\begin{remark}
\label{rem:invariance}
  \begin{enumerate}
    \item
      Let $K\subseteq X\cap Z$ be convex and closed.
      Then the invariance of $K$ under $T$ is equivalent to the invariance of $K$ under $\lambda(\lambda-A_T)^{-1}$ for all sufficiently large $\lambda>0$.
    \item
      Let $S_Y$ be strongly continuous on $Y$. Then $S_{X+Y}$ is strongly continuous on $X+Y$. Let $A_{X+Y}$ be the generator of $S_{X+Y}$. Then
      $\RprimeSXY(\lambda) = \bigl((\lambda-A_{X+Y})^{-1}\bigr)'|_{(X+Y)'\cap E}$.
      Let $L\subseteq (X+Y)'\cap E$ be a convex cone and closed in $(X+Y)'$. Then $L$ is invariant under $(S_{X+Y})'$ if and only if $L$ is invariant under $\lambda \RprimeSXY(\lambda)$ for all sufficiently large $\lambda>0$.
      
      To prove this we first claim that if $B\in L(X+Y)$ then $L$ is invariant under $B'$ if and only if $L^*$ is invariant under $B$. Indeed, let $L$ be invariant under $B'$.
      Let $L^*:=\set{z\in (X+Y);\; \dupa{e}{z}_{(X+Y)'\times (X+Y)}\geq 0\; (e\in L)}$ be the dual cone. Then for $z\in L^*$, $e\in L$ we obtain
      \[\dupa{e}{Bz}_{(X+Y)'\times (X+Y)} = \dupa{B'e}{z}_{(X+Y)'\times (X+Y)}\geq 0\]
      since $B'e\in L$. Hence, $Bz\in L^*$, i.e.\ $L^*$ is invariant under $B$.
      If $L^*$ is invariant under $B$, then analogously $(L^*)^* = \set{e\in (X+Y)';\; \dupa{e}{z}_{(X+Y)'\times (X+Y)}\geq 0\;(z\in L^*)}$ is invariant under $B'$. By the bipolar theorem, $(L^*)^* = L$ since $L$ is a closed convex cone.     
      
      By the claim, we now observe that invariance of $L$ under $(S_{X+Y})'$ is equivalent to invariance of $L^*$ under $S_{X+Y}$.
      Since $S_{X+Y}$ is a $C_0$-semigroup on $X+Y$, invariance of $L^*$ under $S_{X+Y}$ is equivalent to invariance of $L^*$ under $\lambda(\lambda-A_{X+Y})^{-1}$ for sufficiently large $\lambda>0$.
      Again by the claim, invariance of $L^*$ under $\lambda(\lambda-A_{X+Y})^{-1}$ is equivalent to invariance of $L$ under $\lambda\bigl((\lambda-A_{X+Y})^{-1}\bigr)'$. Since $L\subseteq (X+Y)'\cap E$, this is equivalent to invariance of $L$ under $\lambda\RprimeSXY(\lambda)$. 
    \item
      Let $E$ be invariant under $S_Y'$ and under $\lambda R_{S_Y}(\lambda)'$ for all sufficiently large $\lambda>0$ (by means of embedding $E$ into $E''$ canonically first),
      $L\subseteq (X+Y)'\cap E$ a convex cone and closed in $\sigma((X+Y)'\cap E,X+Y)$ and invariant under $(S_{X+Y})'$.
      Then $L$ is invariant under $\lambda \RprimeSXY(\lambda)$ for all sufficiently large $\lambda>0$.
      Indeed, let $L$ be invariant under $(S_{X+Y})'$. Let $L^*$ be the dual cone. For $e\in L$ and $z\in L^*$ we obtain
      \[\dupa{\lambda \RprimeSXY(\lambda)e}{z}_{(X+Y)'\cap E \times (X+Y)} = \int_0^\infty \lambda e^{-\lambda t} \dupa{S_{X+Y}(t)'e}{z}_{(X+Y)'\cap E \times (X+Y)}\, dt \geq 0\]
      since $(S_{X+Y})'(t)e \in L$. Thus, $\lambda \RprimeSXY(\lambda)e \in (L^*)^* = L$ by the bipolar theorem and the fact that $L$ is a closed convex cone.
  \end{enumerate}
\end{remark}

In some applications one can estimate the integrand at the right-hand side of statement (a) in Theorem \ref{thm:main3}.
Under such an assumption the formulation of the theorem simplifies and yields generalized versions of the perturbation result in \cite{bgk2009,w2011,SeifertWingert2014}.

\begin{corollary}
\label{cor:main} 
  Assume Hypothesis~\ref{hyp:mainresult}.
  Let $L\subseteq (X+Y)'\cap E$, $L$ invariant under $(S_{X+Y})'$ and under $\lambda \RprimeSXY(\lambda)$,
  $K\subseteq X\cap Z$, $K$ invariant under $T$ and under $\lambda(\lambda-A_T)^{-1}$ for all sufficiently large $\lambda>0$.
  Let $B\in L(Z,Y)$. 
  Furthermore, assume there exist $M\geq 1$, $\omega\in\R$ such that for all $t\geq 0$ and all $x \in K$, $v'\in L$ we have
  \begin{equation}
  \label{eq:extra_assumption}
    \dupa{S_Y(t-s)BT_Z(s)x}{v'}_{Y\times E} \leq M e^{\omega t} \dupa{Bx}{v'}_{Y\times E} \quad(0\leq s\leq t).
  \end{equation}
  Then the following are equivalent:
  \begin{enumerate}
    \item
      There exists $C_1\geq 0$ such that for all $t\geq 0$ and all $x \in K$, $v'\in L$ we have
      \[\dupa{T(t)x}{v'}_{X\times X'} \leq \dupa{S(t)x}{v'}_{X\times X'} + C_1 te^{\omega t}\dupa{Bx}{v'}_{Y\times E}.\]
    \item
      There exists $C_2\geq 0$ such that for all sufficiently large $\lambda>\omega$ and all $ x\in K$, $v'\in L$ we have
      \[\dupa{(\lambda-A_T)^{-1}x}{v'}_{X\times X'} \leq \dupa{(\lambda-A_S)^{-1}x}{v'}_{X\times X'} + \frac{C_2}{(\lambda-\omega)^2}\dupa{Bx}{v'}_{Y\times E}.\]
    \item
      There exists $C_3\geq 0$ such that for all $u\in D(A_T)\cap K$ and all $v'\in L\cap D(A_S')$ we have
      \[\dupa{A_T u}{v'}_{X\times X'} \leq \dupa{u}{A_S'v'}_{X\times X'} + C_3\dupa{Bu}{v'}_{Y\times E}.\]
  \end{enumerate}
\end{corollary}

\begin{proof}
  The proof of ``(a) $\Rightarrow$ (b)'' and ``(b) $\Rightarrow$ (c)'' is analogous to (but easier than) the proof of Theorem~\ref{thm:main}, so we only sketch the differences.
  For ``(c) $\Rightarrow$ (a)'' we make use of the corresponding implication in Theorem~\ref{thm:main}.
  
  (a) $\Rightarrow$ (b): Let $\lambda>\omega$ sufficiently large, $ x\in K$, $v'\in L$. Then
  \begin{align*}
    \dupa{(\lambda-A_T)^{-1}x - (\lambda-A_S)^{-1}x}{v'}_{X\times X'} & = \dupa{\int_0^\infty e^{-\lambda t} \bigl(T(t)x-S(t)x\bigr)x\,dt}{v'}_{X\times X'} \\
    & = \int_0^\infty e^{-\lambda t} \dupa{\bigl(T(t)x-S(t)x\bigr)x}{v'}_{X\times X'}\,dt \\
    & \leq  \int_0^\infty e^{-\lambda t} C_1 te^{\omega t} \dupa{Bx}{v'}_{Y\times E}\, dt \\
    & = C_1 \int_0^\infty t e^{-(\lambda-\omega) t}\, dt  \dupa{Bx}{v'}_{Y\times E}.
  \end{align*}
  Integrating by parts we obtain the assertion.
  
  (b) $\Rightarrow$ (c): For sufficiently large $\lambda>\omega$ let again $J_S(\lambda) := \lambda(\lambda-A_S)^{-1}$ and $J_T(\lambda) := \lambda(\lambda-A_T)^{-1}$.
  Let $u\in D(A_T)\cap K$ and $v'\in L\cap D(A_S')$. As in the proof of Theorem~\ref{thm:main} we obtain
  \[A_TJ_T(\lambda)u - A_SJ_S(\lambda)u = \lambda\bigl(J_T(\lambda)u-J_S(\lambda)u\bigr),\]
  which by (b) yields
  \begin{align*}
    \dupa{A_TJ_T(\lambda)u - A_SJ_S(\lambda)u}{v'}_{X\times X'} & = \dupa{\lambda\bigl(J_T(\lambda)u-J_S(\lambda)u\bigr)}{v'}_{X\times X'} \\
    & \leq C_2 \frac{\lambda^2}{(\lambda-\omega)^2} \dupa{Bu}{v'}_{Y\times E},
  \end{align*}
  and therefore
  \[\dupa{J_T(\lambda)A_T u}{v'}_{X\times X'} \leq \dupa{J_S(\lambda)u}{A_S'v'}_{X\times X'} + C_2 \frac{\lambda^2}{(\lambda-\omega)^2}\dupa{B u}{v'}_{Y\times E}.\]
  Again, $\lambda\to \infty$ yields (c).
  
  (c) $\Rightarrow$ (a): 
  By Theorem~\ref{thm:main} ``(c) $\Rightarrow$ (a)'' with $B$ replaced by $C_3B$, for all $t\geq 0$ and all $x \in K$, $v'\in L$ we obtain
  \[\dupa{T(t)x}{v'}_{X\times X'} \leq \dupa{S(t)x}{v'}_{X\times X'} + C_3\int_0^t \dupa{S(t-s)BT(s)x}{v'}_{Y\times E}\, ds\]
  Now, \eqref{eq:extra_assumption} yields
  \begin{align*}
    \dupa{T(t)x}{v'}_{X\times X'} & \leq \dupa{S(t)x}{v'}_{X\times X'} + C_3 \int_0^t M e^{\omega t} \dupa{Bx}{v'}_{Y\times E}\, ds\\
    & = \dupa{S(t)x}{v'}_{X\times X'} + C_3 M t e^{\omega t} \dupa{Bx}{v'}_{Y\times E}. \qedhere
  \end{align*}
\end{proof}

\begin{remark}
\label{rem:strong_formulation}
  In the context of Banach lattices and positive semigroups and positive $B$ one can ask whether the `tested' inequalities in (a) and (b) of Theorem~\ref{thm:main3} and Corollary~\ref{cor:main} can also be
  formulated in a `strong' inequality. It turns out that this is indeed the case provided $L$ contains sufficiently many elements to detect positivity and, for Theorem~\ref{thm:main3}(a), the integral can be computed in $Y$.
\end{remark}

\section{Applications}\label{sec:applications}

\subsection{Bi-continuous semigroups}

Let $Y$ be a Banach space and let $\mathcal{T}_{\norm{\cdot}}$ denote the norm-topology on $Y$.
Let $\mathcal{T}\subseteq \mathcal{T}_{\norm{\cdot}}$ be a locally convex Hausdorff topology on $Y$, such that $(Y,\mathcal{T})' \subseteq (Y,\mathcal{T}_{\norm{\cdot}})'$ is norming for $Y$ and $(Y,\mathcal{T})$ is sequentially complete on norm-bounded subsets of $X$.

Let $S$ be a semigroup on $Y$. We say that $S$ is \emph{locally bi-continuous} if
\begin{enumerate}
  \item
    there exists $M\geq 1$, $\omega\in\R$ such that $\norm{S(t)}\leq Me^{\omega t}$ for all $t\geq 0$,
  \item
    for all $y\in Y$ we have $\calTlim_{t\to 0} S(t)y = y$,
  \item
    for all bounded sequences $(y_n)$ in $Y$ and $y\in Y$ such that $\calTlim_{n\to\infty} y_n = y$ we have $\calTlim_{n\to\infty} S(t)y_n = S(t)y$ uniformly for $t$ in compact subsets of $[0,\infty)$.
\end{enumerate}

For further information on bi-continuous semigroups, see e.g.\ \cite{Kuehnemund2001,Kuehnemund2003}.

\begin{example}
  Let $Y:=C_{\rm b}(\R^d)$ and $\mathcal{T}$ be the compact-open topology on $Y$. Let $k\from (0,\infty)\times\R^d\to \R$, 
  \[k(t,x):=k_t(x):=\frac{1}{(4\pi t)^{d/2}} e^{-\norm{x}^2/(4t)} \quad(t>0, x\in\R^d)\]
  be the Gau{\ss}-Weierstra{\ss} kernel. Define $S_Y\from [0,\infty)\to L(X)$ by
  \[S_Y(t)f := \begin{cases}
		f & t=0,\\
		k_t * f & t>0.
	    \end{cases}\]
  Then $S_Y$ is a locally bi-continuous semigroup on $Y$.
\end{example}

In our abstract setting, we would like to have $S_Y$ to be locally bi-continuous and $E:=(Y,\mathcal{T})'$. However, $(Y,\mathcal{T})'$ is in general not closed. 
We can nevertheless work with this space since the closedness is only used in Lemma~\ref{lem:exponential_bound} to obtain an exponential bound for the semigroup, and such a bound exists by definition.
Note that in such a case the mapping $\dupa{S_Y(\cdot)y}{e}_{Y\times E}$ is continuous for all $y\in Y$ and $e\in (Y,\mathcal{T})'$ by property~(b) and continuity of $e$.
Moreover, by \cite[Section 1.2]{Kuehnemund2001}, there exists a unique operator $A_{S_Y}$ in $Y$ (the \emph{generator} of the bi-continuous semigroup) such that $R_{S_Y}(\lambda) = \kappa(\lambda-A_{S_Y})^{-1}$, where $\kappa\from Y\to E'$ is the embedding.

Hence, the following theorem is essentially a modification of Theorem~\ref{thm:main3} to this setting.

\begin{theorem}
\label{thm:main_bicontinuous}
  Let $X,Y,Z$ be Banach spaces, $V$ a Hausdorff topological vector space, $\mathcal{T}$ a locally convex Hausdorff topology on $Y$ as above. Let:
  \begin{enumerate}\renewcommand{\labelenumi}{{\rm(\roman{enumi})}}
   \item $(X,Y)$, $(X,Z)$ be $V$-compatible;
   \item $S$ $C_0$-semigroup on $X$ (with generator $A_S$);
   \item $T$ $C_0$-semigroup on $X$ (with generator $A_T$);
   \item there exists a $C_0$-semigroup $T_Z$ on $Z$ consistent with $T$;
   \item $E:=(Y,\mathcal{T})'$;
   \item there exists a locally bi-continuous semigroup $S_Y$ on $Y$ consistent with $S$ with generator $A_{S_Y}$.
  \end{enumerate}
  Let $L\subseteq (X+Y)'\cap E$ be invariant under $S'$, $\lambda\bigl((\lambda-A_S)^{-1}\bigr)'$ and under $\lambda\bigl((\lambda-A_{S_Y})^{-1}\bigr)'$ for all sufficiently large $\lambda>0$,
  $K\subseteq X\cap Z$ invariant under $T$ and under $\lambda(\lambda-A_T)^{-1}$ for all sufficiently large $\lambda>0$.
  Let $B\in L(Z,Y)$. 
  Then the following are equivalent:
  \begin{enumerate}
    \item
      For all $t\geq 0$ and all $ x \in K$, $v'\in L$ we have
      \[\dupa{T(t)x}{v'}_{X\times X'} \leq \dupa{S(t)x}{v'}_{X\times X'} + \int_0^t \dupa{S_Y(t-s)BT_Z(s)x}{v'}_{Y\times E}\, ds.\]
    \item
      For all sufficiently large $\lambda\in\R$ and all $ x\in K$, $v'\in L$ we have
      \[\dupa{(\lambda-A_T)^{-1}x}{v'}_{\!X\times X'} \!\leq\! \dupa{(\lambda-A_S)^{-1}x}{v'}_{\!X\times X'} \!+\! \dupa{(\lambda-A_{S_Y})^{-1}B(\lambda-A_{T_Z})^{-1}x}{v'}_{\!Y\times E}.\]
    \item 
      For all $u\in D(A_T)\cap K$ and all $v'\in L\cap D(A_S')$ we have
      \[\dupa{A_T u}{v'}_{X\times X'} \leq \dupa{u}{A_S'v'}_{X\times X'} + \dupa{Bu}{v'}_{Y\times E}.\]
  \end{enumerate}
\end{theorem}

\begin{example}
  As a concrete example let $X:=L_2(\R^d)$, $Y:=C_{\rm b}(\R^d)$, $Z:=C_0^1(\R^d)$.
  Let $\mathcal{T}$ be the compact-open topology on $Y$. Let $k\from (0,\infty)\times \R^d\to \R$, 
  \[k(t,x):=k_t(x):=\frac{1}{(4\pi t)^{d/2}} e^{-\norm{x}^2/(4t)} \quad(t>0, x\in\R^d)\]
  be the Gau{\ss}-Weierstra{\ss} kernel. Define $S\from [0,\infty)\to L(X)$ by
  \[S(t)f := \begin{cases}
  	      f & t=0,\\
  	      k_t * f & t>0.
             \end{cases}\]
  Then $S$ defines a positive $C_0$-semigroup on $X$. Moreover, $S_Y\from [0,\infty)\to L(Y)$,
  \[S_Y(t)f := \begin{cases}
  	      f & t=0,\\
  	      k_t * f & t>0,
             \end{cases}\]
  is locally bi-continuous on $Y$ and consistent with $S$.
  Let $0\leq b_1,\ldots,b_d \in C_{\rm b}(\R^d)$, $b:=(b_1,\ldots,b_d)$ and let $T$ be the positive $C_0$-semigroup on $X$ generated by $-\Delta + b\cdot\grad$.
  Note that $-\Delta + b\cdot\grad$ also generates a $C_0$-semigroup $T_Z$ on $Z$ which is consistent with $T$.
  Let $B\in L(Z,Y)$ be defined by $Bu := b_{\rm max} \sum_{j=1}^d \partial_j u$, where $b_{\rm max}:=\sup_{x\in\R^d} \max_{j\in\set{1,\ldots,d}} b_j(x)$.
  Let $K:=\set{u\in L_2(\R^d)\cap C_0^1(\R^d);\; u\geq 0}\subseteq X\cap Z$. Then $K$ is a closed convex cone and invariant under $T$ and therefore also invariant under $\lambda(\lambda-A_T)^{-1}$ for sufficiently large $\lambda>0$, see Remark~\ref{rem:invariance}(a).
  Let $L:=\set{v'\in (X+Y)'\cap E;\; v'\geq 0}$. Then $L$ is invariant under $(S_{X+Y})'$, $\lambda\bigl((\lambda-A_S)^{-1}\bigr)'$ and under $\lambda\bigl((\lambda-A_{S_Y})^{-1}\bigr)'$ for sufficiently large $\lambda>0$.
  For $u\in D(A_T)\cap K$ and $v'\in L\cap D(A_S')$ we obtain
  \begin{align*}
    \dupa{A_T u}{v'}_{X\times X'} & = \dupa{-\Delta u + b\cdot\grad u}{v'}_{X\times X'} \\
    & \leq \dupa{u}{-\Delta v'}_{X\times X'} + \dupa{b_{\rm max}\cdot \grad u}{v'}_{Y\times E}
    = \dupa{u}{A_S'v'}_{X\times X'} + \dupa{Bu}{v'}_{Y\times E}.
  \end{align*}
  Thus, for all $t\geq 0$, $u \in K$ and $v'\in L$ we obtain
  \[\dupa{T(t)u}{v'}_{X\times X'} \leq \dupa{S(t)u}{v'}_{X\times X'} + \int_0^t \dupa{S_Y(t-s)BT_Z(s)u}{v'}_{Y\times E}\, ds.\]
\end{example}

\subsection{Perturbations of positive semigroups by $L_\infty$-functions}

Let $(\Omega,\mu)$ be a localizable measure space, $X:= L_2(\mu)$, $Y:= L_\infty(\mu)$, $Z:=L_1(\mu)$, $E:=L_1(\mu)$ the predual of $Y$.
Let $T,S$ be positive $C_0$-semigroups on $L_2(\mu)$ with generators $A_S$, $A_T$, respectively.

Assume there exist $M\geq 1$, $\omega\in\R$ such that
\begin{align*}
 \norm{S(t)^* u}_1 & \leq Me^{\omega t}\norm{u}_1 \quad (u\in L_2\cap L_1(\mu), t\geq 0),\\
 \norm{T(t) u}_1 & \leq Me^{\omega t}\norm{u}_1 \quad (u\in L_2\cap L_1(\mu), t\geq 0).
\end{align*}
Then the adjoint semigroup $S'$ of $S$ extends to a consistent positive $C_0$-semigroup $S^*_1$ on $L_1(\mu)$ by \cite[Theorem 7]{Voigt1992}, 
so $S_\infty(t):= S^*_1(t)'$ ($t\geq 0$) defines a weak$^*$-continuous semigroup on $L_\infty(\mu)$
such that $S_\infty$ is positive and consistent with $S$. Moreover, $E$ is invariant under $(S_\infty)'$ and under $\lambda R_{S_\infty}(\lambda)'$ for sufficiently large $\lambda>0$.
Also $T|_{L_2\cap L_1(\mu)}$ extends to a $C_0$-semigroup $T_1$ on $L_1(\mu)$ such that $T_1$ is consistent with $T$, see \cite[Theorem 7]{Voigt1992}.

Let $L:=\set{f\in L_2(\mu)\cap L_1(\mu);\; f\geq 0} \subseteq (X+Y)'\cap E$ and $K:=\set{f\in L_2(\mu)\cap L_1(\mu);\; f\geq 0} \subseteq X\cap Z$.
Then $L$ is invariant under $(S_{X+Y})'$ and $L$ is a $\sigma((L_2+L_\infty)'\cap L_1,L_2+L_\infty)$-closed convex cone. Hence, $L$ is also invariant under $\lambda R_{S_{L_2+L_\infty}}'(\lambda)$ for sufficiently large $\lambda>0$, see Remark~\ref{rem:invariance}(c).
Moreover, $K$ is a closed convex cone and invariant under $T$ and, hence, also under $\lambda(\lambda-A_T)^{-1}$ for all sufficiently large $\lambda>0$. 

Let $B\in L(L_1(\mu),L_\infty(\mu))$ be defined by 
\[Bu:= \int_\Omega u\,d\mu \cdot \1_{\Omega}.\]
Then $B$ is positive. 

With this setup we can recover the result of \cite[Theorem 2.1]{SeifertWingert2014} by applying Corollary \ref{cor:main}
and noting that testing against elements of $L$ can detect positivity; cf.\ Remark \ref{rem:strong_formulation}.

\subsection{Perturbations of positive semigroups by $L_p$-functions}

Let $(\Omega,\mu)$ be a localizable measure space, $1\leq p,q<\infty$, $X:= L_2(\mu)$, $Y:= L_p(\mu)$, $Z:=L_q(\mu)$, $E:=L_{p'}(\mu)$ with $\tfrac{1}{p}+\tfrac{1}{p'} = 1$, i.e.\ $E=Y'$. 
Note that $X+Y$ is dense in $X$ and in $Y$, hence $(X+Y)'=X'\cap Y'$.

Let $T,S$ be positive $C_0$-semigroups on $L_2(\mu)$ with generators $A_S$, $A_T$, respectively.

Assume there exist $M\geq 1$, $\omega\in\R$ such that
\begin{align*}
 \norm{S(t) u}_p & \leq Me^{\omega t}\norm{u}_p \quad (u\in L_2\cap L_p(\mu), t\geq 0),\\
 \norm{T(t) u}_q & \leq Me^{\omega t}\norm{u}_q \quad (u\in L_2\cap L_q(\mu), t\geq 0).
\end{align*}
Then $S|_{L_2\cap L_p(\mu)}$ extends to a $C_0$-semigroup $S_p$ on $L_p(\mu)$, and $T|_{L_2\cap L_q(\mu)}$ extends to a $C_0$-semigroup $T_q$ on $L_q(\mu)$, such that
$S_p$ is consistent with $S$ and $T_q$ is consistent with $T$ (for the $L_1$-case see \cite[Theorem 7]{Voigt1992}).

Let $L:= \set{f\in L_2(\mu)\cap L_{p'}(\mu);\; f\geq 0} \subseteq (X+Y)'\cap E$ and $K:=\set{f\in L_2(\mu)\cap L_q(\mu);\; f\geq 0} \subseteq X\cap Z$.
Since $L$ is a closed convex cone and invariant under $S'$, it is also invariant under $\lambda R_{S_{L_2+L_p}}'(\lambda)$ for sufficiently large $\lambda>0$, see Remark~\ref{rem:invariance}(b).
Moreover, $K$ is a closed convex cone and invariant under $T$ and hence also under $\lambda(\lambda-A_T)^{-1}$ for all sufficiently large $\lambda>0$, see Remark~\ref{rem:invariance}(a).

Let $0\leq f\in L_p(\mu)$, $0\leq g'\in L_{q'}(\mu)$, where $\tfrac{1}{q}+\tfrac{1}{q'} = 1$, and define $B\in L(L_q(\mu),L_p(\mu))$ by 
\[Bu:= \int_\Omega ug'\,d\mu \cdot f.\]
Then $B$ is positive. With this setup we precisely recover the result of \cite[Theorem 1]{SeifertWaurick2016}.

\subsection{Delay equations}

We apply our result in the context of $C_0$-semigroups for delay equations. For more information concerning delay equations in the semigroup context, see \cite{BatkaiPiazzera2001,BatkaiPiazzera2005}.
  
Let $\Xs$ be a real Banach lattice (the state space), $1\leq p<\infty$. Let $\Phi_p\from W_p^1\bigl((-1,0);\Xs\bigr)\to \Xs$ be linear and bounded.
Let $A_0$ be the generator of a positive $C_0$-semigroup on $\Xs$. Let $\mathcal{A}_{\Phi_p,p}$ in $\mathcal{X}_p:=\Xs\times L_p\bigl((-1,0); \Xs\bigr)$ be defined by
\begin{align*}
  D(\mathcal{A}_{\Phi_p,p}) & := \set{\begin{pmatrix} x\\f\end{pmatrix} \in D(A_0)\times W_p^1\bigl((-1,0);\Xs\bigr);\; f(0) = x},\\
  \mathcal{A}_{\Phi_p,p} & = \begin{pmatrix} 
				A_0 & \Phi_p \\
				0 & \partial
			      \end{pmatrix}.
\end{align*}

Now, let $1\leq q<\infty$ and $\Psi_q\from W_q^1\bigl((-1,0);\Xs\bigr)\to \Xs$ be linear, bounded and positive,
such that $\Phi_p f\leq \Psi_q f$ for all $0\leq f\in W_p^1\bigl((-1,0);\Xs\bigr)\cap W_q^1\bigl((-1,0);\Xs\bigr)$.

For a concrete example let $\eta\from [-1,0]\to L(\Xs)$ be of bounded variation and $\Phi_p f:= \int_{-1}^0 \,d\eta f$.
Furthermore, assume that $d\eta$ is absolutely continuous with respect to Lebesgue measure $\lambda$, 
$\frac{d\eta}{d\lambda}\leq \rho \in L_{q'}\bigl((-1,0);L(\Xs)\bigr)$, where $\frac{1}{q}+\frac{1}{q'}=1$, and let $\Psi_qf:= \int_{-1}^0 \rho(t)f(t)\,dt$.

We now apply the abstract theory with $X=Y:=\mathcal{X}_p$, $Z:=\mathcal{X}_q:=\Xs\times L_q\bigl((-1,0); \Xs\bigr)$, and $E:=Y'$.
Let $S$ be the positive $C_0$-semigroup on $\mathcal{X}_p$ generated by $\mathcal{A}_{0,p} = \begin{pmatrix} 
				A_0 & 0 \\
				0 & \partial
			      \end{pmatrix}$
(with domain $D(A_0)\times  W_p^1\bigl((-1,0);\Xs\bigr)$), and $T$ be the positive $C_0$-semigroup on $\mathcal{X}_p$ generated by $\mathcal{A}_{\Phi_p,p}$. 
Assume that $T$ is exponentially bounded on $\mathcal{X}_q$ and can thus be extended to a consistent $C_0$-semigroup $T_Z$ on $\mathcal{X}_q$.

Let $B:=\begin{pmatrix} 0 & \Psi_q \\ 0 & 0 \end{pmatrix}\in L(\mathcal{X}_q,\mathcal{X}_p)$.
Furthermore, let $\tilde{B}:=\begin{pmatrix} 0 & \Phi_p \\ 0 & 0 \end{pmatrix}$. Then $\mathcal{A}_{\Phi_p,p} = \mathcal{A}_{0,p} + \tilde{B}$, and we have
$\tilde{B}u\leq Bu$ for all $0\leq u\in \mathcal{X}_p\cap \mathcal{X}_q$.

Let $K:=\set{u\in \mathcal{X}_p\cap \mathcal{X}_q;\; u\geq 0}$ and $L:=\set{v'\in \mathcal{X}_p';\; v'\geq 0}$.
Then $K$ is a closed convex cone and invariant under $T$, hence also invariant under $\lambda(\lambda-\mathcal{A}_{\Phi_p,p})^{-1}$ for sufficiently large $\lambda\in\R$.
Furthermore, $L$ is a closed convex cone and invariant under $S'$, hence also invariant under $\lambda(\lambda-\mathcal{A}_{0,p})^{-1}$ for sufficiently large $\lambda\in\R$.

For $u\in K\cap D(\mathcal{A}_{\Phi_p,p})$ and $v'\in L\cap D(\mathcal{A}_{0,p}')$ we obtain
\[\dupa{\mathcal{A}_{\Phi_p,p} u}{v'}_{\mathcal{X}_p\times \mathcal{X}_p'} = \dupa{(\mathcal{A}_{0,p} + \tilde{B}) u}{v'}_{\mathcal{X}_p\times \mathcal{X}_p'} \leq \dupa{u}{\mathcal{A}_{0,p}' v'}_{\mathcal{X}_p\times \mathcal{X}_p'} + \dupa{B u}{v'}_{\mathcal{X}_p\times \mathcal{X}_p'}.\]
By Theorem~\ref{thm:main3}, for $u\in K$ and $v'\in L$ and $t\geq 0$ we observe
\begin{align*}
  \dupa{T(t)u}{v'}_{\mathcal{X}_p\times \mathcal{X}_p'} & \leq \dupa{S(t)u}{v'}_{\mathcal{X}_p\times \mathcal{X}_p'} + \int_0^t \dupa{S(t-s)BT_Z(s)u}{v'}_{\mathcal{X}_p\times \mathcal{X}_p'}\,ds \\
  & = \dupa{S(t)u+ \int_0^t S(t-s)BT_Z(s)u\,ds }{v'}_{\mathcal{X}_p\times \mathcal{X}_p'}.
\end{align*}
Since testing against functionals in $L$ detects positivity in $\mathcal{X}_p$, we obtain
\begin{equation}T(t)u\leq S(t)u+\int_0^t S(t-s)BT_Z(s)u\,ds\label{eq:impl}\end{equation}
for all $u\in K$ and $t\geq 0$.

\begin{remark}
  Of course, the delay semigroup $T$ is given by
  \begin{equation}T(t) = S(t) + \int_0^t S(t-s)\tilde{B}T(s)\,ds \quad (t\geq0).\label{eq:voc}\end{equation}
  We stress that we work with different spaces $\mathcal{X}_p$ and $\mathcal{X}_q$, and our framework deals exactly with that case. In particular, for this reason \eqref{eq:impl} does \emph{not} immediately follow from \eqref{eq:voc}.
\end{remark}

\section*{Acknowledgements}

C.S.\ thanks Jan Meichsner for useful discussions on bi-continuous semigroups.
We also thank J\"urgen Voigt for stimulating discussions on earlier versions of this manuscript.

\noindent
Christian Seifert \\
Technische Universit\"at Hamburg \\
Institut f\"ur Mathematik \\
21073 Hamburg, Germany \\
{\tt christian.se\rlap{\textcolor{white}{hugo@egon}}ifert@tu\rlap{\textcolor{white}{darmstadt}}hh.de}

\medskip

\noindent
Hendrik Vogt \\
Universit\"at Bremen \\
Fachrichtung Mathematik \\
01062 Bremen, Germany \\
{\tt hendrik.vo\rlap{\textcolor{white}{hugo@egon}}gt@uni-\rlap{\textcolor{white}{darmstadt}}bremen.de}

\medskip

\noindent
Marcus Waurick \\
University of Strathclyde\\
Department of Mathematics and Statistics\\
G1 1XH, Scotland\\
{\tt marcus.wau\rlap{\textcolor{white}{hugo@egon}}rick@strath.\rlap{\textcolor{white}{darmstadt}}ac.uk}

\end{document}